\documentclass[11pt,letterpaper,reqno]{amsart}
\usepackage[T1]{fontenc}
\usepackage{lmodern}
\usepackage[margin=1in]{geometry}
\usepackage{amsmath,amssymb,mathtools}
\usepackage[expansion=false]{microtype}
\usepackage{xcolor}
\usepackage[colorlinks=true,linkcolor=blue!50!black,citecolor=green!35!black,urlcolor=blue!60!black]{hyperref}
\newtheorem{theorem}{Theorem}[section]
\newtheorem{lemma}[theorem]{Lemma}
\newtheorem{proposition}[theorem]{Proposition}
\newtheorem{corollary}[theorem]{Corollary}
\theoremstyle{remark}

\numberwithin{equation}{section}
\newcommand{\R}{\mathbb R}
\newcommand{\C}{\mathbb C}
\newcommand{\E}{\mathbb E}
\newcommand{\supp}{\operatorname{supp}}
\newcommand{\norm}[1]{\left\lVert#1\right\rVert}

\newcommand{\eps}{\varepsilon}
\newcommand{\cR}{\mathcal R}
\newcommand{\cG}{\mathcal G}
\newcommand{\cT}{\mathcal T}
\allowdisplaybreaks[2]
\title[Inverse random source scattering] {Inverse source problem for the stochastic fractional Helmholtz equation in inhomogeneous media}
\author[D. Qiu]{Dong Qiu}
  \address{School of Mathematical Sciences, Zhejiang University, Hangzhou 310058, China}
\email{qiudong@zju.edu.cn}

  \author[X. Xu]{Xiang Xu}
  \address{School of Mathematical Sciences, and Center for Interdisciplinary Applied Mathematics, Zhejiang University, Hangzhou 310058, China}
\email{xxu@zju.edu.cn}
  
 \author[Y. Zhao]{Yue Zhao}
\address{School of Mathematics and Statistics, and Key Lab NAA-MOE, Central China Normal University,
Wuhan 430079, China}
\email{zhaoy@ccnu.edu.cn}

\date{}
\subjclass[2020]{35R30, 35R11, 35R60, 35P25}
\keywords{fractional Helmholtz equation, inverse source problem, white noise, single frequency, stability}
\begin{document}
\begin{abstract}
This paper is concerned with an inverse source problem for the stochastic fractional Helmholtz equation in a three-dimensional inhomogeneous medium.
The source is modeled as real Gaussian spatial white noise with an unknown nonnegative strength function.
For admissible frequencies, we establish well-posedness of the radiating solution in the distributional sense and show that it has a smooth representation outside the source region.
For the inverse problem, by using two-point correlations of the field measured in an exterior open set at a single frequency, we prove a logarithmic stability estimate for recovering the source strength under an a priori Sobolev bound.
As a uniform H\"older stability estimate is unavailable for the general admissible class, by assuming the difference between two source strengths has a fixed sign, we improve the stability to H\"older continuity with an optimal Hölder exponent.
The analysis combines scattering theory and quantitative continuation for the fractional Laplacian using the Caffarelli-Silvestre extension.
\end{abstract}
\maketitle

\section{Introduction}

The fractional Helmholtz operator arises in fixed-energy models of fractional quantum mechanics \cite{Laskin1} and in descriptions of spatial dispersion \cite{TT}.
Its nonlocality couples the field across separated points.
In an inhomogeneous medium, the refractive coefficient and potential also affect wave propagation.
When the source fluctuates randomly, correlations of the radiated field provide information about its statistical properties.

We study the inverse source problem for the stochastic fractional Helmholtz equation
\begin{equation}\label{eq:model}
\bigl(( -\Delta)^\sigma-k^{2\sigma}r(x)+q(x)\bigr)u(x,k)
=\sqrt{\mu(x)}\,\dot W(x)\quad\text{in }\mathbb R^3,
\end{equation}
where $0<\sigma<1$, $k>0$, $r$ is a refractive coefficient, and $q$ is a potential.
The fractional Laplacian is defined by the Fourier multiplier
\[
 \widehat{(-\Delta)^\sigma v}(\xi)=|\xi|^{2\sigma}\widehat v(\xi).
\]
The functions $1-r$ and $q$ are compactly supported, while $\dot W$ denotes real Gaussian spatial white noise.
The nonnegative function $\mu$ describes the strength of the random source.
We assume that the source and the inhomogeneities are contained in $K$ with $K\Subset B$, where $B$ is a bounded region.
We seek the outgoing field satisfying the Sommerfeld radiation condition:
\begin{equation}\label{src}
 \lim_{\rho\to\infty}\rho(\partial_\rho u-iku)=0,
 \qquad \rho=|x|,
\end{equation}
uniformly in the direction $\hat x=x/|x|$.

Our objective is to determine $\mu$ from measurements made in an open set $O$ outside $B$.
More precisely, the data are the two-point correlations
\begin{equation}\label{eq:near-data}
F(x,z;k)=\mathbb E[u(x,k)u(z,k)],\qquad x,z\in O,
\end{equation}
at a fixed frequency $k$.
The geometric assumptions are specified at the beginning of Section~\ref{dp} and apply throughout the paper.
Because the source is spatial white noise, its individual realizations are distributions rather than ordinary functions.
Establishing the meaning and regularity of the radiating solution is therefore a necessary part of the inverse analysis.

The known coefficients $r,q\in C^\infty(\R^3;\R)$ are fixed and satisfy
\[
 r\geq r_0>0,\qquad
 \supp(r-1)\cup\supp q\subset K.
\]

Our main result is formulated as follows.
Let $u_1,u_2$ be the solutions corresponding to two sources $\sqrt{\mu_1}\dot W$ and $\sqrt{\mu_2}\dot W$, respectively.
The data discrepancy is defined by
\begin{equation}\label{eq:full-error}
 \delta=\norm{F_1-F_2}_{L^2(O\times O)},
\end{equation}
where $F_i = \E[u_i(x)u_i(z)], i=1,2$.
For $m>3/2$ and $E>0$, introduce the a priori class for the strength functions
\begin{equation}\label{eq:class}
 \mathcal A=\{\mu\in H^m(\R^3):\mu\geq0,\ \supp\mu\subset K,
                         \ \norm\mu_{H^m}\leq E\}.
\end{equation}
According to Theorem \ref{thm:direct}, the direct problem for the fractional Helmholtz equation with variable coefficients is well posed for all $k>0$ except for $k\in\Sigma\subset (0, +\infty)$, where $\Sigma$ consists of discrete poles.
Thus, for the study of the inverse problem we assume $k\notin\Sigma$.

The following stability estimate is the main result of this paper, which implies the uniqueness of the inverse problem.

\begin{theorem}\label{thm:pair-log}
Assume $k\notin\Sigma$.
Let $0<\sigma<1$ and $m>3/2$.
For $\mu_1,\mu_2\in\mathcal A$, the following stability estimate holds
\begin{equation}\label{eq:pair-log}
 \norm{\mu_1-\mu_2}_{L^\infty}
 \leq CE\left(\log\left(3+\frac E\delta\right)\right)^{-\theta},
\end{equation}
where
\begin{equation}\label{eq:pair-log-exponent}
 \theta=\theta_0\frac{m-3/2}{\max\{m-1+\sigma,1\}}>0.
\end{equation}
Here $C>0$ is a constant, and $\theta_0$ may be any fixed exponent with $0<\theta_0<\sigma/\kappa$, where $\kappa>0$ is the constant specified by Lemma~\ref{lem:local-log}.
\end{theorem}

The logarithmic stability obtained above in general cannot be improved to a H\"older one.
To obtain a H\"older stability, we assume that the difference of the strength functions has a fixed sign.
As a consequence, we obtain the following H\"older stability.

\begin{theorem}\label{thm:ordered-holder}
Assume $0<\sigma<1$, $m>3/2$, and $k\notin\Sigma$.
Let $\mu_1,\mu_2\in\mathcal A$ satisfy either $\mu_1\geq\mu_2$ or $\mu_2\geq\mu_1$.
One has the stability estimates
\begin{equation}\label{eq:ordered-holder}
 \norm{\mu_1-\mu_2}_{L^1}\leq C\delta\quad {\rm and}\quad
 \norm{\mu_1-\mu_2}_{L^\infty}
 \leq CE^{1-\beta}\delta^\beta,
\end{equation}
where $C$ is a generic constant and
\[
 \beta=\frac{m-3/2}{m+3/2}.
\]
Moreover, the exponent $\beta$ is optimal.
\end{theorem}

Unique continuation has played a central role in inverse problems for fractional operators.
The fractional Calder\'on problem and its stability have been studied in \cite{GSU,RS,Ruland,Covi}.
For deterministic fractional Helmholtz equations, simultaneous recovery of a source and a scattering potential was studied in \cite{CL}.
High-frequency inverse scattering in variable refractive backgrounds was investigated in \cite{CDS}, and deterministic potential recovery from far-field scattering amplitudes was obtained in \cite{msq}.

The Caffarelli--Silvestre extension \cite{CS} realizes $(-\Delta)^\sigma$ as the weighted Dirichlet-to-Neumann operator of a local degenerate elliptic equation in one higher dimension.
This local--nonlocal correspondence extends to fractional powers of second-order operators \cite{ST}.
In fractional inverse problems, the extension has supported Carleman-based unique continuation \cite{RulandUCP}, partial-data recovery through Runge approximation \cite{GSU,GLX}, and quantitative propagation of smallness for stability \cite{RS,Ruland}.
It also provides a route, under the variable-coefficient assumptions of \cite{CGRU}, from partial nonlocal Calder\'on data to the full local Dirichlet-to-Neumann map.
Related extension-based reductions for elliptic inverse source and parabolic Calder\'on problems appear in \cite{LinSource,LLU}.
In the present proof, we use the extension to propagate smallness of an auxiliary field from the exterior observation region toward the source region.

For classical stochastic Helmholtz equations, inverse random-source scattering was studied in \cite{BCL2016}; stability for white-noise sources, including in inhomogeneous media, was established in \cite{LL}; and fixed-frequency conditional stability was obtained in \cite{MH}.
In the fractional setting, \cite{LiLi} recovers principal symbols of the covariance and relation operators of a microlocally isotropic generalized Gaussian source from a single realization of high-frequency far-field data.
Its model includes a known deterministic potential supported away from the source.
For a homogeneous fractional white-noise model, \cite{SE} proves single-frequency uniqueness for the variance and obtains increasing stability from multifrequency exterior correlations.

In this work, the unknown is the strength $\mu$ of a real spatial white-noise source in a known inhomogeneous medium.
Our aim is to establish a stability estimate for recovering the strength function from correlation data at  a single admissible frequency.
To address the low regularity of the white-noise source, we first develop a direct theory using the outgoing free resolvent and an integral equation formulation.
For admissible frequencies, the solution exists in the distributional sense and has a smooth Green function representation in the exterior observation region.
When $\sigma>3/4$, it also admits a continuous mild realization, so the stochastic integral formulation can be interpreted pointwise throughout the space.
For source strengths in a bounded $H^m(\mathbb R^3)$ class, $m>3/2$, we establish a logarithmic stability estimate in $L^\infty$ for recovering the source strength from exterior correlations at a fixed frequency, which also implies uniqueness.
We also show that a uniform H\"older stability estimate does not hold for the full admissible class.
If the difference of two strengths has a fixed sign, we obtain a Lipschitz estimate in $L^1$ and a H\"older estimate in $L^\infty$ with an optimal exponent.
The analysis combines scattering theory and
quantitative continuation for the fractional Laplacian using the Caffarelli--Silvestre extension.


The paper is organized as follows.
Section \ref{dp} establishes the direct problem and the continuous mild realization.
Section \ref{ip} proves the stability estimates and examines their sharpness.
In the appendix we collect the continuation and approximation results used in the inverse analysis.

\section{Direct problem}\label{dp}

Throughout the paper, including the statements in the Introduction, $B$ is a bounded open set, $K\Subset B$ is compact with nonempty interior, and $O$ is a nonempty bounded open subset of $\R^3\setminus\overline B$.
These assumptions imply $\operatorname{dist}(K,\overline O)>0$.
For the inverse results, the source strengths belong to $\mathcal A$ in \eqref{eq:class}.
We use the Fourier transform $\widehat f(\xi)=\int_{\R^3}e^{-ix\cdot\xi}f(x)\,dx$, whose inverse carries the factor $(2\pi)^{-3}$.
For the inverse estimates, $k\notin\Sigma$ is fixed.
Their constants may depend on $\sigma,m,k,r,q,K,B,O$ and the chosen $\theta_0$.
The exponent $\theta_0$ is fixed with $0<\theta_0<\sigma/\kappa$, where $\kappa$ is provided by Lemma~\ref{lem:local-log}, with the observation ball chosen once and for all inside $O$.

We make precise the distributional interpretation and the outgoing solution class.
A candidate solution is a tempered distribution $u$ whose Fourier transform is locally integrable near $\xi=0$.
We require $u$ to be smooth outside a ball and to satisfy $u(x)=O(|x|^{-1})$ and \eqref{src}, uniformly in $x/|x|$.
Near zero the product $|\xi|^{2\sigma}\widehat u$ is ordinary function multiplication; away from zero it is distributional multiplication by a smooth symbol.
A smooth cutoff combines these definitions, which agree on their overlap.
All distributional solutions and uniqueness assertions below are understood in this class.
For random fields, these conditions hold on one event of probability one.

\subsection{Well-posedness of the direct problem}
For the free equation ($r=1$, $q=0$), the outgoing resolvent is defined by
\[
 \cR_0(k)=\bigl((-\Delta)^\sigma-k^{2\sigma}-i0\bigr)^{-1}.
\]
Here $-i0$ denotes the distributional limit of the inverse with $-i\varepsilon$ as $\varepsilon\downarrow0$, selecting the outgoing solution.
Its convolution kernel $G_{\sigma,k}$ has the following properties (cf. \cite[Lemma~2.1]{SE}).
\begin{lemma}\label{lem:free}
Let $0<\sigma<1$ and $k>0$.
Then
\begin{align}
 G_{\sigma,k}(x)&=\frac{k^{2-2\sigma}}{4\pi\sigma}
       \frac{e^{ik|x|}}{|x|}+J_{\sigma,k}(|x|),
                         \label{eq:green-split}\\
 J_{\sigma,k}(d)&=\frac{\sin(\pi\sigma)}{4\pi^2d}
  \int_0^\infty\frac{t^\sigma e^{-d\sqrt t}}
   {t^{2\sigma}-2k^{2\sigma}t^\sigma\cos(\pi\sigma)+k^{4\sigma}}\,dt.
                         \label{eq:green-tail}
\end{align}
The kernel is smooth away from zero and locally integrable, with
\begin{equation}\label{eq:singularity}
 G_{\sigma,k}(x)=d_\sigma|x|^{2\sigma-3}(1+o(1))\quad(|x|\to0),\qquad
 d_\sigma=\frac{\sin(\pi\sigma)\Gamma(2-2\sigma)}{2\pi^2}>0.
\end{equation}
Every derivative of $J_{\sigma,k}(|x|)$ of order $j$ is $O(|x|^{-3-2\sigma-j})$ at infinity.
For $\chi_1,\chi_2\in C_c^\infty(\R^3)$ and $t\in\R$, the localized resolvent operator
\begin{equation}\label{eq:free-local-map}
 \chi_1\cR_0(k)\chi_2:H^t(\R^3)\longrightarrow H^{t+2\sigma}(\R^3)
\end{equation}
is bounded.
Moreover, the localized operator $\chi_1\cR_0(z)\chi_2$ extends holomorphically to $\Re z>0$ in the complex plane using the principal branch of $z^{2\sigma}$.
\end{lemma}

\begin{proof}
The Green kernel decomposition follows from the fractional-power resolvent identity in \cite[(5.28)]{MS}; the explicit formulas and their local and far-field estimates are given in \cite[Lemma~2.1]{SE}.
We give the operator argument for the last assertions.
Extend the kernel decomposition to $\Re z>0$ using the principal branch.
The Fourier multiplier of its tail term is
\[
 \frac{\sin(\pi\sigma)}{\pi}\int_0^\infty
 \frac{t^\sigma\,dt}
 {(t+|\xi|^2)
  (t^{2\sigma}-2t^\sigma z^{2\sigma}\cos(\pi\sigma)+z^{4\sigma})}.
\]
On every compact subset of this half-plane the denominator's second factor has modulus at least $c(t^{2\sigma}+|z|^{4\sigma})$.
Indeed, it factors as $(z^{2\sigma}-t^\sigma e^{i\pi\sigma}) (z^{2\sigma}-t^\sigma e^{-i\pi\sigma})$, whereas $|\arg z^{2\sigma}|<\pi\sigma$.
Splitting the integral at $t=|\xi|^2$ for $|\xi|\geq1$, and using $t^{\sigma-1}$ as an integrable majorant near zero otherwise, we bound this multiplier by $C\langle\xi\rangle^{-2\sigma}$ locally uniformly in $z$.
The scalar Cauchy formula, with this bound, also gives holomorphy in the operator norm from $H^t$ to $H^{t+2\sigma}$.
The classical Helmholtz resolvent after spatial localization is entire from $H^t$ to $H^{t+2}$.
Its compactly localized kernel is $e^{iz|x-y|}/(4\pi|x-y|)$ and the usual local elliptic estimate gives the stated two-derivative regularity gain.
Multiplication by $z^{2-2\sigma}/\sigma$ preserves holomorphy in the right half-plane.
The sum agrees with the spectral resolvent in the first quadrant and with its outgoing boundary value for real $k>0$.
This proves the localized mapping property and continuation.
The proof is complete.
\end{proof}

Set
\begin{equation}\label{eq:contrast}
 V=q+k^{2\sigma}(1-r).
\end{equation}
For a compactly supported source $f$, the corresponding Lippmann--Schwinger equation is
\[
 u=\cR_0(k)(f-Vu).
\]
The outgoing resolvent for the variable-coefficient equation is denoted by $\cR(k)=(( -\Delta)^\sigma+q-k^{2\sigma}r)^{-1}$.

The Lippmann--Schwinger formulation and discrete exceptional frequencies for deterministic fractional Helmholtz scattering are also studied in \cite[Theorems~4.1 and~4.2]{ZCV}.
In the following lemma, we prove the mapping and uniqueness statements needed here for compactly supported distributional sources.
\begin{lemma}\label{lem:fredholm}
There is a discrete set $\Sigma\subset(0,\infty)$ such that, for $k\notin\Sigma$, every compactly supported source has a unique outgoing solution in the class specified above.
The resolvent satisfies
\begin{align}\label{sm}
 \cR(k):H^t_{\rm comp}(\R^3)\longrightarrow
              H^{t+2\sigma}_{\rm loc}(\R^3),\qquad t\in\R.
\end{align}
The bounds are uniform when the source support and the localization sets are fixed.
Its Green kernel $\cG_k(x,y)$ is symmetric in $x,y$ and smooth away from $x=y$.
\end{lemma}

\begin{proof}
Choose $\chi\in C_c^\infty(\R^3)$ equal to one near $\overline B$, and for $\Re z>0$ put
\[
 T=I+\chi\cR_0(z)V,\qquad V=q+z^{2\sigma}(1-r).
\]
For every $s\in\R$, Lemma~\ref{lem:free} shows that $T-I$ maps $H^s$ boundedly into $H^{s+2\sigma}$, with output supported in $\supp\chi$.
The compact Sobolev embedding therefore makes it a holomorphic compact family on $H^s$.
In particular, $T$ is Fredholm of index zero on each of these spaces.

First take $z$ in the first quadrant and work on $L^2$.
If $Tv=0$, then $v=-\chi\cR_0(z)Vv$, and repeated use of the $2\sigma$-derivative gain shows that $v$ is smooth and compactly supported.
Put $w=-\cR_0(z)Vv$.
Since $\chi w=v$ and $\chi=1$ on $\supp V$,
\[
 ((-\Delta)^\sigma+q-z^{2\sigma}r)w=0.
\]
The free spectral resolvent maps $L^2$ to $H^{2\sigma}$ globally there, so $w\in H^{2\sigma}(\R^3)$.
Multiplying the equation by $\overline w$, integrating over $\R^3$, and taking imaginary parts gives
\[
 -\Im(z^{2\sigma})\int_{\R^3}r|w|^2\,dx=0.
\]
Since $\Im(z^{2\sigma})>0$ and $r\geq r_0>0$, this implies $w=v=0$.
Thus $T$ is invertible there.
The analytic Fredholm theorem gives a meromorphic inverse on the right half-plane.
We denote by $\Sigma$ its positive real poles.

At any $k\notin\Sigma$, a vector in the kernel of $T$ on an arbitrary $H^s$ is compactly supported and also belongs to $L^2$ by the repeated regularity gain argument.
Its vanishing follows from the $L^2$ inverse.
Index zero then gives a bounded inverse on every $H^s$.
These inverses agree on their common domains, since their difference lies in the kernel on the lower of the two Sobolev spaces.
Consequently the same exceptional set works for every Sobolev order.

For $f\in H^t$ with compact support, define
\[
 v=T^{-1}\chi\cR_0(k)f,\qquad
 u=\cR(k)f=\cR_0(k)(f-Vv),
\]
using the inverse on $H^{t+2\sigma}$.
The defining equation gives $\chi u=v$ and hence $Vu=Vv$.
These identities and Lemma~\ref{lem:free} prove \eqref{sm}.
We also have $u=\cR_0(k)(f-Vu)$.
The effective source $f-Vu$ is a compactly supported distribution.
Its Fourier transform is smooth and of polynomial growth, so
\[
 \widehat u=(|\xi|^{2\sigma}-k^{2\sigma}-i0)^{-1}
                   (\widehat{f-Vu})
\]
is locally integrable near zero and verifies the equation in the specified distributional sense.
The kernel decomposition in Lemma~\ref{lem:free}, including source-variable derivatives of any fixed order, gives uniformly for $y$ in a fixed compact set
\[
 \partial_y^\alpha G_{\sigma,k}(x-y)=O(|x|^{-1}),\qquad
 \partial_y^\alpha(\partial_{|x|}-ik)G_{\sigma,k}(x-y)
                                      =O(|x|^{-2}).
\]
Pairing these estimates with the compactly supported finite-order distribution $f-Vu$ proves that $u$ belongs to the outgoing class.

Conversely, let $u$ be any outgoing solution with compactly supported source $f$.
Then
\[
 h=u-\cR_0(k)(f-Vu)
\]
belongs to the same class and satisfies $(( -\Delta)^\sigma-k^{2\sigma})h=0$.
Near zero this Fourier equation is an equation of locally integrable functions and yields $\widehat h=0$ there.
Away from zero, it gives $\supp\widehat h\subset\{|\xi|=k\}$.
In a neighborhood of this sphere the quotient
\[
 \frac{|\xi|^2-k^2}{|\xi|^{2\sigma}-k^{2\sigma}}
\]
extends smoothly and is nonzero.
Multiplication by this quotient therefore yields $(-\Delta-k^2)h=0$ on all of $\R^3$.
Thus $h$ is a smooth entire Helmholtz solution satisfying the Sommerfeld condition, and classical Rellich uniqueness gives $h=0$.
Every outgoing solution hence satisfies the Lippmann--Schwinger equation and its restriction $\chi u$ satisfies the invertible equation for $T$, proving uniqueness.

To prove reciprocity, take $f,g\in C_c^\infty(\R^3)$ and first let $z$ lie in the first quadrant.
The fields $\cR(z)f$ and $\cR(z)g$ belong to $H^{2\sigma}(\R^3)$.
Bilinear integration by parts for the even Fourier symbol $|\xi|^{2\sigma}$ and for the multiplication operator $q-z^{2\sigma}r$ gives $\int f\cR(z)g=\int g\cR(z)f$.
Both sides are meromorphic functions of $z$, by the localized construction, so continuation gives
\begin{equation}\label{eq:reciprocity}
 \int_{\R^3} f\,\cR(k)g\,dx
   =\int_{\R^3}g\,\cR(k)f\,dx,
 \qquad f,g\in C_c^\infty(\R^3),\quad k\notin\Sigma.
\end{equation}
This is a bilinear identity, and implies the symmetry of the distribution kernel.

Finally, let $\chi_1,\chi_2$ be compactly supported smooth cutoffs with disjoint supports, take $f=\chi_2 h$, and set $u=\cR(k)f$.
Choose a cutoff $\psi$ supported away from $\supp\chi_2$ and equal to one near $\supp\chi_1$.
The Lippmann--Schwinger equation gives
\[
 u=\cR_0(k)f-\cR_0(k)(V\psi u)
                     -\cR_0(k)(V(1-\psi)u).
\]
Near $\supp\chi_1$, the first and third terms have smooth kernels and their Sobolev norms of every order are controlled by a fixed negative Sobolev norm of $h$, using \eqref{sm} for $Vu$.
The middle term gains $2\sigma$ derivatives from the regularity of $u$ on $\supp\psi$.
Repeating this argument with nested cutoffs inside $\R^3\setminus\supp\chi_2$ gives
\[
 \chi_1\cR(k)\chi_2:H^{-N}(\R^3)\longrightarrow H^M(\R^3)
\]
boundedly for every $M,N>0$, starting with the bound \eqref{sm}.
The Schwartz kernel theorem and Sobolev embedding give joint smoothness of $\cG_k(x,y)$ off the diagonal as claimed.
\end{proof}

The real white noise is an isonormal Gaussian process on real $L^2(B)$, extended complex linearly for complex integrands.
Thus
\[
 \E[W(h)W(j)]=\int_B h j,\qquad
 \E[W(h)\overline{W(j)}]=\int_Bh\overline j.
\]
Throughout this section, let \(\mu\in L^\infty(\mathbb R^3)\) be nonnegative with \(\operatorname{supp}\mu\subset K\).
As a consequence of Lemma~\ref{lem:fredholm}, we establish the well-posedness of the direct scattering problem with white noise source.
\begin{theorem}\label{thm:direct}
Let $0<\sigma<1$ and $k\notin\Sigma$.
For the given real white noise, there exists a unique outgoing random distribution solving \eqref{eq:model}--\eqref{src} in the class specified above.
It is the centered Gaussian random distribution $u=\cR(k)(\sqrt\mu\dot W)$ and, on one event of probability one,
\[
 u\in H^t_{\rm loc}(\R^3)\qquad
          \hbox{for every }t<2\sigma-\frac32.
\]
The solution has a smooth modification on $\R^3\setminus K$ given by
\begin{equation}\label{eq:stochastic-green}
 u(x,k)=\int_K\cG_k(x,y)\sqrt\mu(y)\,dW(y),\qquad x\notin K.
\end{equation}
\end{theorem}

\begin{proof}
Define the source on test functions by $\langle\sqrt\mu\dot W,\phi\rangle=\int_B \sqrt\mu\phi\,dW$.
For an orthonormal basis $\{e_j\}$ of real $L^2(B)$, the white-noise isometry and Parseval's identity give, for every $b>3/2$,
\[
 \sum_j\norm{\sqrt\mu e_j}_{H^{-b}(\R^3)}^2
 =\frac{\norm {\sqrt{\mu}}_{L^2(B)}^2}{(2\pi)^3}
       \int_{\R^3}(1+|\xi|^2)^{-b}\,d\xi<\infty.
\]
Thus the Gaussian series with coefficients $\int_Be_j\,dW$ converges in $L^2(\Omega;H^{-b}(\R^3))$ and almost surely in $H^{-b}$, and defines $\sqrt\mu\dot W$ as a distribution supported in $K$.
The realizations for a countable sequence $b\downarrow3/2$ agree as distributions.
Lemma~\ref{lem:fredholm} then constructs $u$, showing that it is centered Gaussian and gives the stated local regularity for all $t<2\sigma-3/2$ using a countable exhaustion of $\R^3$.
It also gives for each compactly supported smooth cutoff $\chi$
\[
 \E\norm{\chi u}_{H^{2\sigma-b}}^2\leq C_{b,\chi}\norm {\sqrt\mu}_{L^2(B)}^2.
\]

On every compact set disjoint from $K$, all derivatives of the kernel in \eqref{eq:stochastic-green} are uniformly square integrable in $y$ after multiplication by $\sqrt\mu$.
The isometry therefore gives finite expected Sobolev norms of every integer order on compact exterior subsets.
Sobolev embedding and a countable exhaustion yield a common smooth modification.
Stochastic Fubini, first against a countable dense family of test functions, identifies this field with the distribution $\cR(k)(\sqrt\mu\dot W)$.

Pathwise, the effective source $\sqrt\mu\dot W-Vu$ is compactly supported and has finite order.
The free-kernel estimates in the proof of Lemma~\ref{lem:fredholm} give $u=O(|x|^{-1})$ and \eqref{src} on the same probability-one event.
The deterministic uniqueness in Lemma~\ref{lem:fredholm} applied for each realization of the prescribed white noise proves the uniqueness.
\end{proof}

\subsection{The continuous mild solutions}

The stochastic representation in Theorem~\ref{thm:direct} is defined outside $K$ for every $0<\sigma<1$.
We now construct a continuous realization on $\R^3$ when $\sigma>3/4$, following the constructive approach of \cite[Theorem~2.7]{BCL2016} and its fractional adaptation in \cite[Section~2.2]{SE}.
For convenience we denote
\[
W(h):=\int_Bh\,dW.
\]

We first give a precise approximation of the random source.
Recall that a bounded operator $T:H_1\to H_2$ between Hilbert spaces is Hilbert--Schmidt if
\[
 \norm T_{\mathrm{HS}}^2:=\sum_{j=1}^{\infty}
       \norm{Te_j}_{H_2}^2<\infty
\]
for an orthonormal basis $\{e_j\}$ of $H_1$.
The sum does not depend on the basis.
Choose a real-valued orthonormal basis $\{e_j\}$ of $L^2(B)$.
In the present setting, the map $J:L^2(B)\to H^{-b}(\R^3)$, $Jh=\sqrt{\mu}h$, satisfies for $b>3/2$
\begin{equation}\label{eq:noise-HS}
 \norm J_{\mathrm{HS}}^2
 =\frac{\norm {\sqrt{\mu}}_{L^2(B)}^2}{(2\pi)^3}
       \int_{\R^3}(1+|\xi|^2)^{-b}\,d\xi<\infty.
\end{equation}
Indeed, using Fourier transformation, Tonelli's theorem and Parseval's identity in $L^2(B)$ we have
\[
 \sum_j\norm{\sqrt{\mu}e_j}_{H^{-b}}^2
 =\frac1{(2\pi)^3}\int_{\R^3}(1+|\xi|^2)^{-b}
   \sum_j\left|\int_B\sqrt{\mu}(y)e_j(y)e^{-iy\cdot\xi}\,dy\right|^2d\xi,
\]
which equals the right-hand side of \eqref{eq:noise-HS}.
Hence, the partial sums $\sum_{j=1}^n W(e_j)Je_j$ converge in $L^2(\Omega;H^{-b}(\R^3))$ as $n\to\infty$.
The pairing of the random source $\sqrt{\mu}\dot W$ with a test function $\phi$ is $W(\sqrt\mu\phi)$, which defines $\sqrt{\mu}\dot W$ as a random distribution.

\begin{proposition}\label{prop:noise-approx}
Let $\{B_{n,j}\}_{j=1}^{N_n}$ be finite measurable partitions of $B$ into sets of positive measure such that $\max_j\operatorname{diam}(B_{n,j})\to0$.
Define
\[
 \dot W_n=\sum_{j=1}^{N_n}
       \frac{W(\mathbf1_{B_{n,j}})}{|B_{n,j}|}\mathbf1_{B_{n,j}},
 \qquad u_n=\cR(k)(\sqrt\mu\dot W_n),
\]
where $\mathbf1_{B_{n,j}}(x)$ is the characteristic function, and $W(\mathbf1_{B_{n,j}}) = \int_{B_{n, j}}{\rm d}W$.
For every $b>3/2$ and $\eta\in C_c^\infty(\R^3)$ we have
\begin{equation}\label{eq:noise-approx-H}
 \E\norm{\eta(u_n-u)}_{H^{2\sigma-b}(\R^3)}^2\longrightarrow0.
\end{equation}
Moreover, for every compact $Q\Subset\R^3\setminus K$ and integer $\ell\geq0$ we have
\begin{align}\label{con}
 \E\norm{u_n-u}_{C^\ell(Q)}^2\longrightarrow0.
\end{align}
\end{proposition}

\begin{proof}
Since
\[
 \E\norm{\sqrt{\mu}\dot W_n}_{L^2(B)}^2
 =\sum_{j=1}^{N_n}\frac1{|B_{n,j}|}\int_{B_{n,j}}\mu(y)\,dy<\infty,
\]
each regularized random source is an $L^2$ function almost surely.
Let $P_n$ be the orthogonal projection in $L^2(B)$ onto the cellwise constant functions.
Uniform continuity first gives $P_nh\to h$ for every continuous function $h$ on $\overline B$.
Then by density and $\norm{P_n}\leq1$ we obtain the strong convergence on $L^2(B)$.
The white-noise isometry implies
\[
 \E\norm{\sqrt{\mu}\dot W_n-\sqrt{\mu}\dot W}_{H^{-b}}^2
 =\norm{J(P_n-I)}_{\mathrm{HS}}^2
 =\norm{(P_n-I)J^*}_{\mathrm{HS}}^2\longrightarrow0.
\]
To see the last limit, we take the sum over an orthonormal basis $\{f_j\}$ of $H^{-b}(\R^3)$.
For each $j$, the term $\norm{(P_n-I)J^*f_j}_{L^2(B)}^2$ tends to zero and is bounded by $\norm{J^*f_j}_{L^2(B)}^2$, whose sum is finite by \eqref{eq:noise-HS}.
The localized resolvent estimates in Lemma~\ref{lem:fredholm} prove \eqref{eq:noise-approx-H}.
For the exterior estimate, choose $\zeta\in C_c^\infty(\R^3)$ equal to one near $K$ and supported away from $Q$.
Off-diagonal smoothness of the Green function gives uniform $H^b$ bounds for $\zeta(y)\partial_x^\alpha\cG_k(x,y)$ with $x\in Q$, $|\alpha|\leq\ell$.
Duality therefore bounds the exterior $C^\ell$ norm by the $H^{-b}$ norm of the source difference, which proves \eqref{con}.
\end{proof}

\begin{proposition}\label{prop:mild}
Let $3/4<\sigma<1$ and $k\notin\Sigma$.
The solution in Theorem~\ref{thm:direct} has a continuous modification satisfying, on one event of probability one,
\begin{equation}\label{eq:mild}
 u(x)=\int_KG_{\sigma,k}(x-y)\sqrt{\mu(y)}\,dW(y)
       -\int_B G_{\sigma,k}(x-y)V(y)u(y)\,dy,
 \qquad x\in\R^3.
\end{equation}
For the same white noise, this field is unique up to indistinguishability among continuous random fields satisfying \eqref{eq:mild}.
The regularized solutions $u_n$ in Proposition~\ref{prop:noise-approx} and $u$ have common continuous versions such that, for every bounded cube $Q$, $1\leq p<\infty$, and $0<\gamma<2\sigma-3/2$,
\begin{equation}\label{eq:mild-strong-approx}
 \E\norm{u_n-u}_{C^{0,\gamma}(Q)}^p\longrightarrow0.
\end{equation}
Moreover,
\begin{equation}\label{eq:mild-moments}
 \sup_n\E\norm{u_n}_{C^{0,\gamma}(Q)}^p
       +\E\norm u_{C^{0,\gamma}(Q)}^p
 \leq C\norm{\sqrt\mu}_{L^\infty}^p.
\end{equation}
The constant depends only on the fixed $\sigma,k,r,q,K,B,Q,p$, and $\gamma$.
\end{proposition}

\begin{proof}
Let $P_n$ be the orthogonal projections from the proof of Proposition~\ref{prop:noise-approx}, extended complex linearly.
They satisfy $\norm{P_n}\leq1$ and $P_nh\to h$ in $L^2(B)$; the partitions need not be nested.
Set
\[
 h_x(y)=\sqrt{\mu(y)}G_{\sigma,k}(x-y),\qquad
 w_n(x)=W(P_nh_x),\qquad w(x)=W(h_x).
\]
A direct computation on the partition cells gives $w_n=\cR_0(k)(\sqrt\mu\dot W_n)$.
In this identity $P_n$ acts on the whole product $\sqrt{\mu(y)}G_{\sigma,k}(x-y)$.
The kernel estimates in \cite[Lemma~2.1]{SE} give, on every bounded cube $Q$ and for $0<\beta<2\sigma-3/2$,
\begin{equation}\label{eq:mild-kernel-estimates}
 \sup_{x\in Q}\norm{G_{\sigma,k}(x-\cdot)}_{L^2(B)}\leq C,
 \qquad
 \norm{G_{\sigma,k}(x-\cdot)-G_{\sigma,k}(z-\cdot)}_{L^2(B)}
       \leq C_\beta|x-z|^\beta.
\end{equation}
Thus $\sup_{x\in Q}\norm{h_x}_2\leq C\norm{\sqrt\mu}_\infty$ and $\norm{h_x-h_z}_2\leq C_\beta\norm{\sqrt\mu}_\infty|x-z|^\beta$.
The complex linear extension of the real white noise satisfies $\E|W(h)|^2=\norm h_2^2$ and, for each finite $s\geq2$, $\E|W(h)|^s\leq C_s\norm h_2^s$.

Fix $p$ and $\gamma$ as in the statement.
Choose $\gamma<\beta<2\sigma-3/2$, take a finite $s\geq\max\{2,p\}$ sufficiently large, and then choose $t$ such that
\[
 \gamma+3/s<t<\beta<1.
\]
Each $w_n$ is a finite Gaussian linear combination of deterministic $\beta$-H\"older functions on $Q$, and is therefore a strongly measurable $W^{t,s}(Q)$-valued random variable.
For every $x,z\in Q$, strong convergence of $P_n$ gives convergence to zero of the $s$-th moments of both $(w_n-w_j)(x)$ and $(w_n-w_j)(x)-(w_n-w_j)(z)$ as $n,j\to\infty$.
Moreover,
\[
 \begin{split}
 \E\left|(w_n-w_j)(x)-(w_n-w_j)(z)\right|^s
 &\leq C_s\norm{(P_n-P_j)(h_x-h_z)}_2^s\\
 &\leq C\norm{\sqrt\mu}_\infty^s|x-z|^{s\beta}.
 \end{split}
\]
The corresponding integrand in the fractional Sobolev seminorm is bounded by $C\norm{\sqrt\mu}_\infty^s|x-z|^{s(\beta-t)-3}$, which is integrable on $Q\times Q$.
The pointwise moments are uniformly bounded by $C\norm{\sqrt\mu}_\infty^s$.
Tonelli's theorem and dominated convergence therefore imply
\[
 \E\norm{w_n-w_j}_{W^{t,s}(Q)}^s\longrightarrow0,
 \qquad
 \sup_n\E\norm{w_n}_{W^{t,s}(Q)}^s
       \leq C\norm{\sqrt\mu}_\infty^s.
\]
Since $W^{t,s}(Q)$ is separable, the sequence converges in $L^s(\Omega;W^{t,s}(Q))$.
The continuous embedding $W^{t,s}(Q)\hookrightarrow C^{0,\gamma}(Q)$, \cite[Theorem~8.2]{DNPV}, gives a continuous representative of its limit.
At each fixed $x$, evaluation of this limit agrees with $W(h_x)$, because $w_n(x)\to W(h_x)$ in $L^2(\Omega)$ by the white-noise isometry.
We use $w$ for this continuous modification.
Consequently,
\[
 \E\norm{w_n-w}_{C^{0,\gamma}(Q)}^p\longrightarrow0,
 \qquad
 \sup_n\E\norm{w_n}_{C^{0,\gamma}(Q)}^p
       +\E\norm w_{C^{0,\gamma}(Q)}^p
       \leq C\norm{\sqrt\mu}_\infty^p,
\]
where the $p$-th moment statements follow from $p\leq s$.
The limits obtained on different cubes and for different admissible parameters agree at a countable dense set and hence everywhere by continuity.
A countable exhaustion by cubes and a sequence of exponents increasing to $2\sigma-3/2$ select common versions.
These random variables are strongly measurable in the H\"older norm: they are continuous images of strongly measurable variables in the separable space $W^{t,s}(Q)$.

It remains to transfer these properties to the perturbed equation.
Let
\[
 \mathcal K_kv=(\cR_0(k)(Vv))|_B.
\]
Its kernel $G_{\sigma,k}(x-y)V(y)$ is square integrable on $B\times B$ by \eqref{eq:mild-kernel-estimates}, so $\mathcal K_k$ is compact on $L^2(B)$.
Lemma~\ref{lem:fredholm} gives a bounded inverse of $I+\mathcal K_k$: an element of its kernel produces the homogeneous outgoing solution $-\cR_0(k)(Vv)$, whose restriction to $B$ is $v$, and hence must vanish.
Set
\[
 \begin{aligned}
 v_n&=(I+\mathcal K_k)^{-1}(w_n|_B),&
 v&=(I+\mathcal K_k)^{-1}(w|_B),\\
 u_n&=w_n-\cR_0(k)(Vv_n),&
 \widetilde u&=w-\cR_0(k)(Vv).
 \end{aligned}
\]
The expression for $u_n$ agrees with the regularized solution by the Lippmann--Schwinger equation, and $\widetilde u|_B=v$ in $L^2(B)$ almost surely.
Cauchy--Schwarz and \eqref{eq:mild-kernel-estimates} give
\[
 \norm{\cR_0(k)(Vv)}_{C^{0,\gamma}(Q)}
       \leq C\norm v_{L^2(B)},\qquad
 \norm v_{L^2(B)}\leq C\norm w_{L^2(B)}.
\]
In particular, $\widetilde u$ is continuous and satisfies \eqref{eq:mild}.
Choose a larger bounded cube $Q'$ containing $Q\cup\overline B$.
The same estimates, applied to differences, yield pathwise
\[
 \begin{split}
 \norm{u_n-\widetilde u}_{C^{0,\gamma}(Q)}
 &\leq\norm{w_n-w}_{C^{0,\gamma}(Q)}
                      +C\norm{w_n-w}_{L^2(B)}\\
 &\leq C\norm{w_n-w}_{C^{0,\gamma}(Q')}.
 \end{split}
\]
The analogous bounds for $u_n$ and $\widetilde u$ prove \eqref{eq:mild-strong-approx}--\eqref{eq:mild-moments} with $\widetilde u$ in place of $u$.
These deterministic bounded linear maps also preserve strong measurability.

Finally, take $b=2\sigma>3/2$ in Proposition~\ref{prop:noise-approx}.
It gives $u_n\to u$ in $L^2(\Omega;L^2(Q))$ for every bounded cube, while the preceding argument gives convergence to $\widetilde u$ in the same space.
Thus $\widetilde u=u$ as random distributions, and Theorem~\ref{thm:direct} supplies the equation and the outgoing condition on one event of probability one.
If two continuous fields satisfy \eqref{eq:mild} for the same white noise, their difference $d$ satisfies $(I+\mathcal K_k)(d|_B)=0$.
Thus $d=0$ in $B$, and \eqref{eq:mild} then gives $d=0$ everywhere.
A countable dense set and continuity give equality on one event of probability one.
\end{proof}

\section{Inverse random source problem}\label{ip}

In this section, we establish the stability estimate for the inverse random source problem.
For an open set $W$, let $H^\sigma(W)$ be the restriction of $H^\sigma(\R^3)$ to $W$, with the quotient norm, and let $H^{-\sigma}(W)$ be the dual of $\widetilde H^\sigma(W)$ which is the closure of $C_c^\infty(W)$ in $H^\sigma(\R^3)$.

The following local continuation estimate is the main ingredient for deriving the stability.
Its proof uses the Cauchy interpolation estimate and the interior three-ball inequality for the Caffarelli--Silvestre extension in \cite[Propositions~5.13 and~5.4]{RS}.
\begin{lemma}\label{lem:local-log}
Let $0<\sigma<1$, let $B\subset\R^3$ be a bounded open set, and let $W\subset\R^3$ be a nonempty open ball.
There exists $\kappa>0$, depending only on $\sigma,B,W$, with the following property.
Suppose that $z\in H^\sigma(\R^3)$ satisfies
\begin{equation}\label{eq:local-log-assumptions}
 \norm z_{H^\sigma(\R^3)}\leq M,\qquad
 \norm z_{H^\sigma(W)}+
 \norm{(-\Delta)^\sigma z}_{H^{-\sigma}(W)}\leq\eps,
\end{equation}
where $M>0$ and $\eps\geq0$.
For every fixed $0<\theta_0<\sigma/\kappa$,
\begin{equation}\label{eq:local-log-estimate}
 \norm z_{L^2(B)}\leq CM
 \left(\log\left(3+\frac M\eps\right)\right)^{-\theta_0}.
\end{equation}
The right-hand side is interpreted as zero when $\eps=0$.
The constant depends only on $\sigma,B,W$ and the fixed $\theta_0$.
\end{lemma}

\begin{proof}
The estimates below apply to complex-valued functions by treating their real and imaginary parts separately.
Let $Z$ be the Caffarelli--Silvestre extension \cite{CS} of $z$ given by
\[
 Z(\cdot,t)=P_{\sigma,t}*z,\quad t>0,
\]
where
\[
P_{\sigma,t}(x)=c_\sigma
 \frac{t^{2\sigma}}{(|x|^2+t^2)^{(3+2\sigma)/2}}.
\]
Here the constant $c_\sigma$ is chosen such that $\int_{\R^3}P_{\sigma,t}\,dx=1$.
By the extension identities in \cite[Section~4]{RS} we have
\[
 \operatorname{div}(t^{1-2\sigma}\nabla Z)=0\quad\hbox{in }\R^4_+
\]
and
\[
\qquad Z(\cdot,0)=z,\qquad
 \lim_{t\to 0}t^{1-2\sigma}\partial_tZ
       =-a_\sigma(-\Delta)^\sigma z,
\]
where $a_\sigma$ is a positive constant.
The trace and conormal limits hold in $L^2(\R^3)$ and $H^{-\sigma}(\R^3)$, respectively.
The unit-mass kernel gives $\sup_{t>0}\norm{Z(\cdot,t)}_{L^2}\leq\norm z_{L^2}$, while the extension energy identity gives
\[
\int_{\R^4_+}t^{1-2\sigma}|\nabla Z|^2=a_\sigma\norm{(-\Delta)^{\sigma/2}z}_{L^2}^2.
\]
Consequently, for each fixed $H>0$,
\begin{equation}\label{eq:extension-energy-v5}
 \norm{t^{(1-2\sigma)/2}\nabla Z}_{L^2(\R^4_+)}
 +\norm{t^{(1-2\sigma)/2}Z}_{L^2(\R^3\times(0,H))}
 \leq C_HM.
\end{equation}

Assume first that $0<\rho:=\eps/M<1$.
Choose a fixed half-ball whose flat base is contained in $W$.
The local Cauchy interpolation estimate of \cite[Proposition~5.13]{RS}, after translation and a fixed rescaling, gives an interior ball $D_0$ and $a\in(0,1)$ such that
\begin{equation}\label{eq:initial-smallness-v5}
 \norm{t^{(1-2\sigma)/2}Z}_{L^2(D_0)}\leq CM\rho^a.
\end{equation}
Here $a$ is an interpolation exponent local to this proof.
Indeed, both boundary terms in that proposition are bounded by $C\eps$, and its bulk term is bounded by \eqref{eq:extension-energy-v5}.
In addition, its assumptions concern only the local extension equation, so no support condition on $z$ is needed.

Next, for $0<h<h_0<1$, we propagate \eqref{eq:initial-smallness-v5} to balls above $B\times[h,2h]$.
First, we move at a fixed positive height and then descend with geometrically decreasing heights.
The balls may be centered at height $5r$ with radius $r$, as in \cite[Proposition~5.4]{RS}.
During descent, take $r_{j+1}=4r_j/5$ with a fixed horizontal center.
The center distance is $r_j$, so the next ball lies in the preceding ball of radius $2r_j$, while each ball of radius $4r_j$ stays above $t=0$.
The balls remain in a fixed slab of the upper half-space, and the number of steps satisfies
\[
 N(h)\leq N_0+N_1\log(1/h).
\]
The constants are uniform over the bounded target set.
The scaled three-ball inequality and \eqref{eq:extension-energy-v5} yield, for every final ball $D_h$ of radius comparable to $h$,
\[
 \norm{t^{(1-2\sigma)/2}Z}_{L^2(D_h)}
 \leq CM\rho^{a\alpha^{N(h)}}
 \leq CM\rho^{c h^\kappa},
\]
where $\alpha\in(0,1)$ and $c,\kappa>0$ are fixed.
Here the multiplicative constants remain bounded because the normalized iteration involves the sum $1+\alpha+\alpha^2+\cdots$.
A covering by at most $Ch^{-3}$ such balls therefore gives
\begin{equation}\label{eq:thin-slab-v5}
 \norm{t^{(1-2\sigma)/2}Z}_{L^2(B\times(h,2h))}
 \leq CMh^{-3/2}\rho^{c h^\kappa}.
\end{equation}

To recover the trace, for almost every $t_0\in[h,2h]$ the weighted Cauchy--Schwarz inequality and \eqref{eq:extension-energy-v5} imply
\begin{equation}\label{eq:vertical-trace-v5}
 \norm{z-Z(\cdot,t_0)}_{L^2(B)}
 \leq\left(\int_0^{t_0}t^{2\sigma-1}\,dt\right)^{1/2}
 \norm{t^{(1-2\sigma)/2}\partial_tZ}_{L^2(B\times(0,t_0))}
 \leq CMh^\sigma.
\end{equation}
This follows first between positive heights and then by the $L^2$ trace limit.
By averaging we also obtain such a $t_0$ with
\[
 \norm{Z(\cdot,t_0)}_{L^2(B)}
 \leq Ch^{\sigma-1}
 \norm{t^{(1-2\sigma)/2}Z}_{L^2(B\times(h,2h))},
\]
which combined with \eqref{eq:thin-slab-v5} yields
\begin{equation}\label{eq:log-optimization-v5}
 \norm z_{L^2(B)}\leq CM
 \left(h^{\sigma-5/2}\rho^{c h^\kappa}+h^\sigma\right).
\end{equation}
For large $L=\log(1/\rho)$, choose
\[
 h=\left(\frac{3\log L}{c\kappa L}\right)^{1/\kappa},
\]
which is less than $h_0$ for all sufficiently large $L$.
The ratio of the first term in \eqref{eq:log-optimization-v5} to the second one is
\[
 h^{-5/2}\exp(-cLh^\kappa)
 =\left(\frac{c\kappa}{3}\right)^{5/(2\kappa)}
   (\log L)^{-5/(2\kappa)}L^{-1/(2\kappa)}.
\]
This ratio is bounded, and hence
\[
 \norm z_{L^2(B)}\leq CM
 \left(\frac{\log L}{L}\right)^{\sigma/\kappa}.
\]
With $L=\log(3+M/\eps)$, the preceding estimate gives
\begin{equation}\label{eq:local-log-refined}
 \norm z_{L^2(B)}\leq CM
 \left(\frac{\log(\mathrm e+L)}{L}\right)^{\sigma/\kappa},
\end{equation}
for small positive $\eps/M$, with a constant depending only on $\sigma,B,W$.
When $\eps/M$ is bounded below, the same estimate follows from $\norm z_{L^2(B)}\leq M$ after increasing $C$.
For each fixed $0<\theta_0<\sigma/\kappa$, combining \eqref{eq:local-log-refined} and the elementary bound
\[
 \left(\frac{\log(\mathrm e+L)}{L}\right)^{\sigma/\kappa}
 \leq C L^{-\theta_0},\qquad L\geq\log3,
\]
we prove \eqref{eq:local-log-estimate}.
The constant in this last bound may increase as $\theta_0$ approaches $\sigma/\kappa$.
If $\eps=0$, apply the positive-error estimates with upper bounds $\eps_j\downarrow0$ and pass to the limit.
The proof is complete.
\end{proof}

Denoting
\[
 b_{\sigma,k}(\xi)=\frac{|\xi|^2-k^2}{|\xi|^{2\sigma}-k^{2\sigma}}\quad (|\xi|\ne k),\qquad
b_{\sigma,k}(\xi)=\frac{k^{2-2\sigma}}\sigma\quad\hbox{if }|\xi|=k,
\]
which is the positive multiplier introduced in Lemma~\ref{lem:outgoing-UC}.
We set $D=-i\nabla$ and define the Fourier multiplier $b_{\sigma,k}(D)$ by
\[
 \widehat{b_{\sigma,k}(D)f}(\xi)
 =b_{\sigma,k}(\xi)\widehat f(\xi).
\]
The following lemma is useful in the subsequent analysis.

\begin{lemma}\label{lem:coercive-recovery}
Let $J\in H^{2-2\sigma}(\R^3)$ with compact support in $B$, and set $Z=b_{\sigma,k}(D)J$.
Then
\begin{equation}\label{eq:source-recovery}
 \norm J_{L^2(\R^3)}\leq k^{2\sigma-2}\norm Z_{L^2(B)}.
\end{equation}
Moreover,
\begin{equation}\label{eq:source-recovery-energy}
 \norm J_{H^{1-\sigma}(\R^3)}\leq C\norm Z_{L^2(B)},
\end{equation}
where $C$ depends only on $\sigma$ and $k$.
\end{lemma}

\begin{proof}
For $t=|\xi|^2/k^2$, the inequalities $t^\sigma\geq t$ on $[0,1]$ and $t^\sigma\leq t$ on $[1,\infty)$ give
\begin{equation}\label{eq:positive-regularizer}
 b_{\sigma,k}(\xi)\geq k^{2-2\sigma}.
\end{equation}
The multiplier is continuous and positive, including at $|\xi|=k$, and is asymptotic to $|\xi|^{2-2\sigma}$ as $|\xi|\to\infty$.
Consequently,
\begin{equation}\label{eq:regularizer-energy-comparison}
 c\langle\xi\rangle^{2-2\sigma}
 \leq b_{\sigma,k}(\xi)
 \leq C\langle\xi\rangle^{2-2\sigma}.
\end{equation}
In particular, the assumed regularity of $J$ gives $Z\in L^2(\R^3)$.
Parseval's identity and the support of $J$ imply
\[
 k^{2-2\sigma}\norm J_{L^2}^2
 \leq(2\pi)^{-3}\int_{\R^3}b_{\sigma,k}|\widehat J|^2\,d\xi
 =\int_B\overline JZ\,dx
 \leq\norm J_{L^2}\norm Z_{L^2(B)},
\]
which proves \eqref{eq:source-recovery}.
Keeping the full lower bound in \eqref{eq:regularizer-energy-comparison} gives
\[
 c\norm J_{H^{1-\sigma}}^2
 \leq(2\pi)^{-3}\int_{\R^3}b_{\sigma,k}|\widehat J|^2\,d\xi
 =\int_B\overline JZ\,dx
 \leq\norm J_{H^{1-\sigma}}\norm Z_{L^2(B)}.
\]
The middle quantity is real and nonnegative.
Dividing when $J\ne0$ proves \eqref{eq:source-recovery-energy} and the case $J=0$ is immediate.
The proof is complete.
\end{proof}

For a source $f$ supported in $K$, we let
\[
 w=\cR(k)f\quad {\rm and}\quad z=(-\Delta-k^2)w.
\]
Recalling $V=q+k^{2\sigma}(1-r)$, denote the effective source by
\[
H=f-Vw.
\]
Then we have $w=\cR_0(k)H$, $z=b_{\sigma,k}(D)H$, and the following identity
\begin{equation}\label{eq:regularized-free-equation}
 ((-\Delta)^\sigma-k^{2\sigma})z=(-\Delta-k^2)H
\end{equation}
in the sense of distributions.

\begin{lemma}\label{lem:compact-smoothing}
Let $0<\sigma<1$, $m>3/2$, and let $W_0,W_1$ be nonempty balls with $W_0\Subset W_1\Subset\R^3\setminus\overline B$.
Suppose that $H\in H^m(\R^3)$ is supported in $K$, that $z=b_{\sigma,k}(D)H$, and that
\[
 \norm H_{H^m}\leq M,\qquad
 \norm z_{H^\sigma(W_1)}\leq\eps,
\]
where $M>0$ and $\eps\geq0$.
Set $a=(2-\sigma-m)_+$ and let $\kappa$ be as in Lemma~\ref{lem:local-log} for $B$ and $W=W_0$.
For every fixed $0<\theta_0<\sigma/\kappa$, one has
\begin{equation}\label{eq:smoothed-source-log}
 \norm H_{H^{1-\sigma}}
 \leq CM\left(\log\left(3+\frac M\eps\right)\right)^
 {-\theta_0(m-1+\sigma)/(m-1+\sigma+a)}.
\end{equation}
The right-hand side is zero when $\eps=0$.
The constant depends on the fixed parameters, balls and $\theta_0$.
\end{lemma}

\begin{proof}
The identity \eqref{eq:regularized-free-equation} gives $(-\Delta)^\sigma z=k^{2\sigma}z$ on $W_1$.
If $a=0$, then $m\geq2-\sigma>2-2\sigma$, so $H\in H^{2-2\sigma}$ and $\norm z_{H^\sigma(\R^3)}\leq CM$.
The local equation bounds the conormal datum in $H^{-\sigma}(W_0)$ by $C\eps$.
Lemma~\ref{lem:local-log} and \eqref{eq:source-recovery-energy}, applied directly to $J=H$, give the asserted estimate in this case; using \eqref{eq:local-log-refined} gives \eqref{eq:smoothed-source-refined} below with $a=0$.

Suppose now that $a>0$.
Then $0<\sigma<1/2$ and $3/2<m<2-\sigma<2$, and in particular $0<m-1+\sigma<1$.
Choose an even nonnegative function $\phi\in C_c^\infty(B_1(0))$ with $\int\phi=1$.
Fix $e_0\in(0,1)$ such that
\[
 K+\overline B_{e_0}(0)\Subset B,\qquad
 \overline W_0+\overline B_{e_0}(0)\Subset W_1.
\]
For $0<e\leq e_0$, set
\[
 \phi_e(x)=e^{-3}\phi(x/e),\qquad
 H_e=\phi_e*H,\qquad z_e=\phi_e*z=b_{\sigma,k}(D)H_e.
\]
Thus $H_e\in C_c^\infty(B)$.
Evenness and unit mass give
\[
 |1-\widehat\phi(e\xi)|
 \leq C\min\{e^2|\xi|^2,1\}
 \leq Ce^{m-1+\sigma}|\xi|^{m-1+\sigma}.
\]
This inequality, the rapid decay of $\widehat\phi$, and \eqref{eq:regularizer-energy-comparison} yield
\begin{equation}\label{eq:smoothing-priors}
 \norm{H-H_e}_{H^{1-\sigma}}\leq CMe^{m-1+\sigma},\qquad
 \norm{z_e}_{H^\sigma(\R^3)}\leq CMe^{-a}.
\end{equation}
For the second bound, the total multiplier order is $\sigma+2-2\sigma=2-\sigma=m+a$.

Choose $\psi\in C_c^\infty(W_1)$ equal to one near $\overline W_0+\overline B_{e_0}(0)$.
On $W_0$, $z_e=\phi_e*(\psi z)$, and
\[
 \norm{\psi z}_{H^\sigma(\R^3)}
 \leq C\norm z_{H^\sigma(W_1)}.
\]
Convolution is uniformly bounded on $H^\sigma$.
Also,
\[
 ((-\Delta)^\sigma-k^{2\sigma})z_e=(-\Delta-k^2)H_e=0
 \quad\hbox{on }W_0,
\]
because $H_e$ remains supported in $B$.
Consequently,
\begin{equation}\label{eq:smoothed-cauchy-data}
 \norm{z_e}_{H^\sigma(W_0)}
 +\norm{(-\Delta)^\sigma z_e}_{H^{-\sigma}(W_0)}\leq C\eps.
\end{equation}

For sufficiently small positive $\eps/M$, let $L=\log(3+M/\eps)$.
Applying Lemma~\ref{lem:local-log} with global bound $CMe^{-a}$ and local bound $C\eps$ gives
\begin{equation}\label{eq:regularized-log}
 \norm{z_e}_{L^2(B)}\leq CMe^{-a}L^{-\theta_0}
\end{equation}
and
\begin{equation}\label{eq:regularized-refined}
 \norm{z_e}_{L^2(B)}\leq CMe^{-a}
 \left(\frac{\log(\mathrm e+L)}{L}\right)^{\sigma/\kappa}.
\end{equation}
Indeed, $e^{-a}\geq1$, and the logarithm corresponding to these two bounds is at least $L$ minus a fixed constant.
The function $\log(\mathrm e+L)/L$ is positive and decreasing for $L>0$; fixed changes in $L$ can therefore be absorbed in $C$ when $L$ is large.

Since $H_e$ is smooth and compactly supported, it satisfies the $H^{2-2\sigma}$ assumption of Lemma~\ref{lem:coercive-recovery}.
Apply \eqref{eq:source-recovery-energy} with $J=H_e$ and $Z=z_e$.
Together with \eqref{eq:smoothing-priors} and \eqref{eq:regularized-log}, this gives
\[
 \norm H_{H^{1-\sigma}}
 \leq CM\bigl(e^{m-1+\sigma}+e^{-a}L^{-\theta_0}\bigr).
\]
Choosing $e=L^{-\theta_0/(m-1+\sigma+a)}\leq e_0$ proves \eqref{eq:smoothed-source-log}.
The same balance with $\Phi(L)=(\log(\mathrm e+L)/L)^{\sigma/\kappa}$ and $e=\Phi(L)^{1/(m-1+\sigma+a)}$ gives the useful refined bound
\begin{equation}\label{eq:smoothed-source-refined}
 \norm H_{H^{1-\sigma}}
 \leq CM\Phi(L)^{(m-1+\sigma)/(m-1+\sigma+a)}.
\end{equation}
For non-small $\eps/M$, both estimates follow from the a priori bound by enlarging $C$.
For $\eps=0$, use positive upper bounds tending to zero.
The proof is complete.
\end{proof}

Now we complete the proof of Theorem~\ref{thm:pair-log}.
\begin{proof}
Write $\eta=\mu_1-\mu_2$ and $\Delta F=F_1-F_2$.
By Lemma~\ref{lem:runge}, choose a fixed $g\in C_c^\infty(O;\C)$ such that
\begin{equation}\label{eq:fixed-probe}
 v=\cT g\quad\hbox{and}\quad \norm{v-1}_{C(K)}\leq\frac14.
\end{equation}
The function $v$ is smooth near $K$ and satisfies $|v|\geq3/4$ on $K$.
Set
\[
 f=\eta v,\qquad w=\cR(k)f.
\]
Then $\norm f_{H^m}\leq CE$.
The bilinear white-noise isometry gives
\begin{equation}\label{eq:covariance}
 F(x,z;k)=\int_K\mu(y)\cG_k(x,y)\cG_k(z,y)\,dy,
 \qquad x,z\in O.
\end{equation}
Reciprocity and Fubini's theorem imply
\begin{equation}\label{eq:contracted-data}
 w(x)=\int_O\Delta F(x,\zeta;k)g(\zeta)\,d\zeta,
 \qquad \norm w_{L^2(O)}\leq C\delta.
\end{equation}

Choose concentric balls $W_0\Subset W_1\Subset O_1\Subset O_2\Subset O$.
The smoothness of the Green kernel off the diagonal gives
\begin{equation}\label{eq:exterior-prior}
 \norm w_{H^4(O_2)}\leq C\norm f_{L^2(K)}\leq CE.
\end{equation}
After multiplication by a cutoff supported in $O_2$ and equal to one near $\overline O_1$, interpolation between $L^2$ and $H^4$ gives
\begin{equation}\label{eq:local-interpolated-data}
 \norm w_{H^{\sigma+2}(O_1)}\leq CE^{1-b}\delta^b,
\end{equation}
where $b=(2-\sigma)/4>0$.
Put $H=f-Vw$ and $z=(-\Delta-k^2)w=b_{\sigma,k}(D)H$.
Lemma~\ref{lem:fredholm} and the smoothness of $V$ imply $\norm H_{H^m}\leq CE$ and $\supp H\subset K$.
Moreover,
\begin{equation}\label{eq:small-cauchy-data}
 \norm z_{H^\sigma(W_1)}\leq CE^{1-b}\delta^b.
\end{equation}
For small positive $\delta/E$, the logarithms $\log(3+c(E/\delta)^b)$ and $\log(3+E/\delta)$ are comparable.
Lemma~\ref{lem:compact-smoothing}, with $M=CE$, therefore yields
\[
 \norm H_{H^{1-\sigma}}\leq CE
 \left(\log\left(3+\frac E\delta\right)\right)^%
 {-\theta_0(m-1+\sigma)/(m-1+\sigma+a)},
 \qquad a=(2-\sigma-m)_+.
\]
The identity $w=\cR_0(k)H$ and the localized mapping property in Lemma~\ref{lem:free} give
\[
 f=H+V\cR_0(k)H,\qquad
 \norm f_{H^{1-\sigma}}\leq C\norm H_{H^{1-\sigma}}.
\]
By continuity, $v$ is nonzero in a neighborhood of $K$.
Choose $\chi\in C_c^\infty(\R^3)$ equal to one near $K$ and supported in that neighborhood.
The function $\chi/v$, extended by zero, is smooth and compactly supported, and $\eta=(\chi/v)f$.
Thus
\begin{equation}\label{eq:pair-L2-log}
 \norm\eta_{L^2}\leq\norm\eta_{H^{1-\sigma}}
 \leq CE\left(\log\left(3+\frac E\delta\right)\right)^%
 {-\theta_0(m-1+\sigma)/(m-1+\sigma+a)}.
\end{equation}

Since $0<1-\sigma<3/2<m$, Fourier decomposition at radius $\rho\geq1$ gives
\[
 \norm\eta_{L^\infty}
 \leq C\rho^{1/2+\sigma}\norm\eta_{H^{1-\sigma}}
      +C\rho^{3/2-m}\norm\eta_{H^m}.
\]
For $\eta\ne0$, choose
\[
 \rho=\left(\frac{\norm\eta_{H^m}}
                    {\norm\eta_{H^{1-\sigma}}}\right)^{1/(m-1+\sigma)}
 \geq1
\]
to balance the two terms.
This gives
\begin{equation}\label{eq:L2-Linfty-interpolation}
 \norm\eta_{L^\infty}
 \leq C\norm\eta_{H^{1-\sigma}}^{(m-3/2)/(m-1+\sigma)}
        \norm\eta_{H^m}^{(1/2+\sigma)/(m-1+\sigma)}.
\end{equation}
The case $\eta=0$ is immediate.
Combining these estimates with $\norm\eta_{H^m}\leq2E$ proves \eqref{eq:pair-log} with
\[
 \theta=\theta_0\frac{m-3/2}{m-1+\sigma+a}
       =\theta_0\frac{m-3/2}{\max\{m-1+\sigma,1\}}.
\]
If $\delta/E$ is bounded below, the result follows from the Sobolev bound after increasing $C$.
If $\delta=0$, use arbitrary positive upper bounds $\delta_j\downarrow0$ in the preceding argument and pass to the limit.
The proof is complete.
\end{proof}

\begin{corollary}\label{cor:pair-log-refined}
Under the assumptions of Theorem~\ref{thm:pair-log}, let $\kappa$ be the constant in Lemma~\ref{lem:local-log} for $B$ and $W=W_0$, where $W_0$ is chosen in the proof of that theorem.
For $\delta>0$, put $L=\log(3+E/\delta)$.
Then
\begin{equation}\label{eq:pair-log-refined}
 \norm{\mu_1-\mu_2}_{L^\infty}
 \leq CE\left(\frac{\log(\mathrm e+L)}{L}\right)^%
 {\frac{\sigma}{\kappa}\frac{m-3/2}{\max\{m-1+\sigma,1\}}}.
\end{equation}
The right-hand side is interpreted as zero when $\delta=0$.
\end{corollary}

\begin{proof}
Repeat the preceding proof using \eqref{eq:smoothed-source-refined}.
For small positive $\delta/E$, both the logarithms $\log(3+c(E/\delta)^b)$ and $\log(3+E/\delta)$, and the values of $\log(\mathrm e+L)/L$ at these logarithms, are comparable.
The resulting exponent on the $H^{1-\sigma}$ bound is $\frac{\sigma}{\kappa}\frac{m-1+\sigma}{m-1+\sigma+a}$.
Interpolation in \eqref{eq:L2-Linfty-interpolation} gives \eqref{eq:pair-log-refined}.
Non-small and zero discrepancies are handled exactly as in that proof.
\end{proof}

The refined estimate implies the pure logarithmic estimates of Theorem~\ref{thm:pair-log} for every fixed $0<\theta_0<\sigma/\kappa$.

We next show that no uniform H\"older estimate holds on the full admissible class.
Specifically, there is no $\alpha>0$ and finite constant $C$ such that
\begin{equation}\label{eq:impossible-holder}
 \norm{\mu_1-\mu_2}_{L^\infty}
 \leq C\norm{F_1-F_2}_{L^2(O\times O)}^\alpha,
 \qquad\mu_1,\mu_2\in\mathcal A.
\end{equation}
To verify this, assume the H\"older stability holds for some constants $\alpha>0$ and $C>0$.
Choose nonnegative $\varphi,\psi\in C_c^\infty(\operatorname{int}K)$ such that $\varphi=1$ on a ball and $\psi\geq\varphi$.
For integers $n\geq1$, let
\[
 \mu_{1,n}=c\big(\psi+n^{-m}\varphi\cos(ny_1)\big),\qquad
 \mu_{2,n}=c\big(\psi-n^{-m}\varphi\cos(ny_1)\big).
\]
These functions are nonnegative.
The Fourier translation formula shows that $\norm{n^{-m}\varphi\cos(ny_1)}_{H^m}$ is bounded independently of $n$.
Choose a fixed $c>0$ sufficiently small such that both strengths belong to $\mathcal A$.
Letting $\eta_n=\mu_{1,n}-\mu_{2,n}$, we have
\begin{equation}\label{eq:oscillating-lower}
 \norm{\eta_n}_{L^\infty}\geq c_0n^{-m}
\end{equation}
for all sufficiently large $n$, since $\cos(ny_1)$ attains the value one in the ball where $\varphi=1$.

By \eqref{eq:covariance}, the correlation difference is
\[
 \Delta F_n(x,\zeta)
 =2c n^{-m}\int_K\cG_k(x,y)\cG_k(\zeta,y)
                    \varphi(y)\cos(ny_1)\,dy.
\]
All source-variable derivatives of the kernel factor are uniformly bounded on $\overline O\times\overline O\times\supp\varphi$.
Repeated integration by parts in $y_1$ therefore gives, for every integer $N\geq1$,
\begin{equation}\label{eq:oscillating-upper}
 \norm{\Delta F_n}_{L^2(O\times O)}\leq C_N n^{-m-N}.
\end{equation}
If \eqref{eq:impossible-holder} held, then \eqref{eq:oscillating-lower}--\eqref{eq:oscillating-upper} would imply $c_0n^{-m}\leq C C_N^\alpha n^{-\alpha(m+N)}$.
Choosing $N$ with $\alpha(m+N)>m$ and letting $n\to\infty$ we obtain a contradiction.

We now prove Theorem~\ref{thm:ordered-holder} under the assumption that $\mu_1-\mu_2$ has a fixed sign.
\begin{proof}
Using Lemma~\ref{lem:runge} we may choose $g\in C_c^\infty(O;\C)$ such that $v=\cT g$ satisfies $\norm{v-1}_{C(K)}\leq1/4$.
Thus, on the set $K$ we have $\Re v\geq3/4$ and $|\Im v|\leq1/4$, which gives $\Re(v^2)\geq1/2$.
Write $\eta=\mu_1-\mu_2$.
Since the sign of $\eta$ is fixed, we obtain
\[
 \frac12\norm\eta_{L^1}
 \leq\left|\int_K\eta(y)v(y)^2\,dy\right|
 =\left|\int_{O\times O}(F_1-F_2)(x,z;k)g(x)g(z)\,dx\,dz\right|
 \leq\norm g_{L^2(O)}^2\delta.
\]
Thus we prove the first estimate in \eqref{eq:ordered-holder}.
Further, combining the Fourier inversion and the estimate $|\widehat\eta|\leq\norm\eta_{L^1}$, we have for $R\geq1$
\[
 \norm\eta_{L^\infty}
 \leq C\bigl(R^3\norm\eta_{L^1}
       +R^{3/2-m}\norm\eta_{H^m}\bigr).
\]
Since $\norm\eta_{H^m}\leq2E$, optimizing $R$ when $\delta/E$ is small gives the second estimate in \eqref{eq:ordered-holder}.
For larger $\delta/E$, this estimate still holds after increasing $C$.

To prove the sharpness of the H\"older exponent, choose $y_0$ in the interior of $K$ and a nonzero smooth function $\varphi\geq0$ in $C_c^\infty(B_1)$.
For small $h>0$, consider the strength functions $ \mu_h$ and $\nu_h$ with
\[
 \mu_h(y)=cE h^{m-3/2}\varphi((y-y_0)/h),
 \qquad \nu_h=0,
\]
where $c>0$ is a constant such that $\norm{\mu_h}_{H^m}\leq E$.
Then $\norm{\mu_h}_{L^\infty}$ is comparable to $ Eh^{m-3/2}$, and $\norm{\mu_h}_{L^1}$ is comparable to $Eh^{m+3/2}$.
Since the Green kernel is smooth and bounded on the separated set $\overline O\times K$, from \eqref{eq:covariance} we have $\norm{F_{\mu_h}-F_{\nu_h}}_{L^2(O\times O)}\leq CEh^{m+3/2}$, where $F_{\mu_h}$ and $F_{\nu_h}$ denote the correlation data with $\mu=\mu_h$ and $\mu=\nu_h$, respectively.
Hence, any H\"older stability estimate with an exponent $\alpha>\beta$ would imply $h^{m-3/2}\leq Ch^{\alpha(m+3/2)}$ for all small $h$, which is a contradiction.
\end{proof}

\section{Conclusion}

We have studied an inverse random source problem for the fractional Helmholtz equation in an inhomogeneous medium.
For admissible frequencies, we established well-posedness of the radiating solution with a spatial white noise source.
The solution has a smooth representation outside the source region and, when the fractional order exceeds $3/4$, admits a continuous mild realization.

Using two-point correlations of the field measured on an exterior open set at a single frequency, we obtained logarithmic stability for the recovery of the source strength under an a priori Sobolev bound.
We have shown that a uniform Hölder estimate is unavailable for the general admissible class.
When the difference between two source strengths has a fixed sign, the stability improves to Lipschitz continuity in $L^1$ and Hölder continuity in $L^\infty$ with an optimal H\"older exponent.
These results show that quantitative recovery is possible from fixed-frequency correlation measurements even when the fractional wave propagates through an inhomogeneous medium.
The analysis combines a resolvent formulation of the direct problem with approximation and quantitative continuation for the fractional Laplacian.

\appendix

\section{Useful lemmas}

We first establish a continuation property for the free outgoing resolvent.
\begin{lemma}\label{lem:outgoing-UC}
Let $F$ be a compactly supported distribution and $w=\cR_0(k)F$.
If $w$ and $(-\Delta)^\sigma w$ vanish on a nonempty open set, then $w=0$ and $F=0$.
\end{lemma}

\begin{proof}
Following the spectral regularization in \cite[Theorem~3.2]{SE}, set $z=(-\Delta-k^2)w$.
Cancellation of the outgoing Fourier pole gives
\begin{equation}\label{eq:regularizer}
 \widehat z(\xi)=b_{\sigma,k}(\xi)\widehat F(\xi),\qquad
 b_{\sigma,k}(\xi)=\frac{|\xi|^2-k^2}{|\xi|^{2\sigma}-k^{2\sigma}}\quad(|\xi|\ne k).
\end{equation}
The value of $ b_{\sigma,k}(\xi)$ on $|\xi|=k$ is $k^{2-2\sigma}/\sigma$.
The multiplier is positive and grows at most polynomially.
Since $\widehat F$ is a smooth function with polynomial growth, we have $z\in H^{-N}(\R^3)$ for some $N>0$.
On the given open set,
\[
 z=0\quad{\rm and}\quad(-\Delta)^\sigma z=(-\Delta-k^2)(-\Delta)^\sigma w=0.
\]
The unique continuation theorem for the fractional Laplacian \cite[Theorem~1.2]{GSU} gives $z=0$.
Positivity of $b_{\sigma,k}$ in \eqref{eq:regularizer} then yields $\widehat F=0$ which gives $F=0$, and hence $w=0$.
Near $\xi=0$, these multiplier products are products of locally integrable functions; away from zero they have their usual distributional meaning.
\end{proof}

For $g\in C_c^\infty(O;\C)$, define
\begin{equation}\label{eq:weights}
 (\cT g)(y)=\int_O\cG_k(x,y)g(x)\,dx=(\cR(k)g)(y).
\end{equation}
\begin{lemma}\label{lem:runge}
Let $k\notin\Sigma$ and let $O$ be a nonempty bounded open subset of $\R^3\setminus\overline B$.
Then
\begin{equation}\label{eq:density}
 \{\cT g|_K:g\in C_c^\infty(O;\C)\}
 \quad\hbox{is dense in }C(K;\C).
\end{equation}
\end{lemma}

\begin{proof}
If the density failed, the Hahn--Banach and Riesz representation theorems would give a nonzero finite complex Borel measure $\nu$ supported in $K$ such that
\[
 \int_K(\cT g)(y)\,d\nu(y)=0
 \qquad\hbox{for every }g\in C_c^\infty(O;\C).
\]
Its bounded Fourier transform implies $\nu\in H^{-b}(\R^3)$ for every $b>3/2$, so $w=\cR(k)\nu$ is well-defined.
Reciprocity and the separated smooth kernel give
\[
 \int_K(\cT g)\,d\nu=\int_O w(x)g(x)\,dx,
\]
and hence $w=0$ on $O$.
Since $\bigl(( -\Delta)^\sigma-k^{2\sigma}r(x)+q(x)\bigr)w=\nu$, $r=1$, and $q=0$ on $O$, we also have $(-\Delta)^\sigma w=k^{2\sigma}w=0$ on $O$.
Moreover, we have
\[
 w=\cR_0(k)(\nu-Vw),
\]
with a compactly supported distribution on the right.
Then Lemma~\ref{lem:outgoing-UC} yields $w=0$, and the equation $\bigl(( -\Delta)^\sigma-k^{2\sigma}r(x)+q(x)\bigr)w=\nu$ gives $\nu=0$, which is a contradiction.
\end{proof}

\section*{Acknowledgements}
 X. Xu is partially supported by the National Key Research and Development Program of China (No. 2024YFA1012303), National Natural Science Foundation of China (No. 12525112), and the Open Research Project of Innovation Center of Yangtze River Delta, Zhejiang University. Y. Zhao is supported by National Natural Science Foundation of China (No. 12371421).

\section*{Conflict of Interest Statement}

The authors declare that they have no conflict of interest.

\section*{Data Availability Statement}
Data sharing is not applicable to this article as no datasets were generated or analyzed during the current study.

\end{document}